\documentclass[11pt,twoside]{article} 

\usepackage[utf8]{inputenc}
\usepackage[T1]{fontenc}
\usepackage[english]{babel}
\usepackage{amsthm}
\usepackage{amsmath}
\usepackage{amssymb}
\usepackage{xcolor}
\usepackage{geometry}
\usepackage{tikz}
\usetikzlibrary{patterns}
\usetikzlibrary{calc}
\usepackage{float}
\usepackage{enumitem}

\newcommand{\fa}{\forall \,}
\newcommand{\ex}{\exists \,}

\renewcommand{\leq}{\,\leqslant\,}
\renewcommand{\geq}{\,\geqslant\,}

\newcommand{\lra}{\longrightarrow}
\newcommand{\lms}{\longmapsto}

\renewcommand{\d}{\mathrm{d}}

\newcommand{\R}{\mathrm{R}}

\newcommand{\NN}{\mathbb{N}}

\newcommand{\RR}{\mathbb{R}}

\newcommand{\ZZ}{\mathbb{Z}}

\newcommand{\AAA}{\mathcal{A}}
\newcommand{\BBB}{\mathcal{B}}
\newcommand{\CCC}{\mathcal{C}}
\newcommand{\DDD}{\mathcal{D}}
\newcommand{\EEE}{\mathcal{E}}

\newcommand{\MMM}{\mathcal{M}}

\theoremstyle{theorem}
\newtheorem{thm}{Theorem}[section]
\newtheorem{lem}[thm]{Lemma}
\newtheorem{claim}{Claim}
\newtheorem{prop}[thm]{Proposition}
\newtheorem{cor}[thm]{Corollary}
\newtheorem*{thm*}{Theorem}
\newtheorem*{lem*}{Lemma}

\theoremstyle{definition}

\newtheorem{exmpl}[thm]{Example}

\theoremstyle{remark}
\newtheorem{rmq}[thm]{Remark}

\renewcommand{\thefootnote}{\fnsymbol{footnote}}

\title{\textbf{On the almost everywhere convergence of two-parameter ergodic averages along directional rectangles}}
\author{Bastien \textsc{Lecluse}}

\begin{document}

	\maketitle

\renewcommand{\thefootnote}{}

\footnote{2020 \emph{Mathematics Subject Classification}: Primary 37A30; Secondary 37A46, 42B25.}

\footnote{\emph{Key words and phrases}: Two-parameter ergodic averages, maximal operators, transfer principle, almost everywhere convergence of ergodic averages.}

\renewcommand{\thefootnote}{\arabic{footnote}}
\setcounter{footnote}{0}

\begin{abstract}
	In this paper, we study the almost everywhere convergence of sequences of two-parameter ergodic averages over rectangles in the plane. On the one hand, we show that if the rectangles we consider have their sides with slopes in a finitely lacunary set, then the averages converge almost everywhere in all $L^p$ spaces, $1 < p < \infty$. On the other hand, given some non-lacunary sets of directions, we construct sequences of rectangles oriented along these directions for which the associated ergodic averages fail to converge almost everywhere in any $L^p$ space, $1 < p < \infty$.
\end{abstract}

	\tableofcontents

\section{Introduction}

Let $T$ be a measure preserving transformation on a probability space $ (X,\AAA,\mu) $. One of the fundamentals of ergodic theory, proved in 1931 by Birkhoff, is the \textit{pointwise ergodic theorem}, also known as \textit{Birkhoff's ergodic theorem}: for any $f \in L^1(X)$,
\[ \lim_{n \to \infty} \frac{1}{n} \sum_{i=0}^{n-1} f(T^ix)  \]
exists for almost every $x \in X$. Since then, many generalizations have been obtained, see for example the results of Bellow, Jones and Rosenblatt in \cite{CVmoving.av.BJR} for ``moving averages''. In this paper, we are interested in the behavior of two-parameter ergodic averages. Let $ S,T $ be a pair of ergodic, commuting, invertible, measure preserving,  transformations on a probability space $ (X,\AAA,\mu) $. Moreover, we assume that $S$ and $T$ are ``jointly aperiodic'' in the sense that for any $(i,j) \in \ZZ^2$, $(i,j) \neq (0,0)$, one has
\[ \mu (\{x \in X : S^iT^jx = x\}) = 0. \]
If $R \subset \RR^2$ is a rectangle (not necessarily parallel to the coordinate axes), let $l(R)$ be the length of its shortest side. Given a sequence $\{R_n\}_{n \in \NN}$ of rectangles of $\RR^2$ with $l(R_n) \to \infty$ as $n \to \infty$, we study the almost everywhere convergence of the averages
\[ M_nf := \frac{1}{\# (R_n \cap \ZZ^2)} \sum_{(i,j) \in R_n} f(S^iT^j) \]
for all $f : X \to \RR$ belonging to a given function space. We will always assume that the sets $R_n$ contain the point $(0,0)$. Such a sequence of rectangles will be called an \textit{averaging process according to $S$ and $T$}, or just a \textit{process} if there is no ambiguity. Some results about the behaviour of two-parameter ergodic averages are already known, especially when averaging on rectangles with sides parallel to the axes. Let us recall some known results. If for all $n \in \NN^*$, $R_n$ is the \textit{square} $\{ 0, \ldots , n-1 \}^2$, then for all $f \in L^1(X)$ the averages
\begin{align}\label{EQ - Lebesgue diff thm}
	\lim_{n \to \infty} \frac{1}{n^2} \sum_{i=0}^{n-1} \sum_{j=0}^{n-1} f(S^iT^jx)
\end{align}
converge for almost every $x \in X$. In 1951, Dunford in \cite{MR42074} and Zygmund in \cite{MR45948} proved independently that averages over rectangles with sides parallel to the coordinate axes converge almost everywhere if $f \in L\log L(X)$: namely under that assumption,
\begin{align}\label{EQ - strong maximal theorem}
	\lim_{m,n \to \infty} \frac{1}{mn} \sum_{i=0}^{m-1} \sum_{j=0}^{n-1} f(S^iT^jx)
\end{align}
exists for almost every $x \in X$. Furthermore, Hagelstein and Stokolos proved in \cite{HagelsteinStokolos2010} that this result is sharp in the sense that given any positive increasing function $\varphi$ on $[0,\infty)$ such that $\varphi(x) = o(\log (x))$ as $x \to \infty$, there exists a function $g \in L\varphi (L)$\footnote{A function $g$ is in $L \varphi(L)(X)$ if the function $t \mapsto |g(t)|\log(|g(t)|)$ is in $L^1(X)$.} for which the limit (\ref{EQ - strong maximal theorem}) fails to converge almost everywhere. They later improved this result in 2011, see \cite{MR2781915}. They generalized it to arbitrary sets of unbounded rectangles with sides parallel to the coordinates axes, more precisely they studied the almost everywhere convergence of the averages
\begin{align}\label{EQ - généralisation strong maximal theorem}
	\lim_{\substack{m,n \to \infty \\ (m,n) \in \Gamma}} \frac{1}{mn} \sum_{i=0}^{m-1} \sum_{j=0}^{n-1} f(S^iT^jx)
\end{align}
where $\Gamma \subset \ZZ_+^2$ is unbounded. ``Moving averages'' over rectangles with sides parallel to the coordinates axes have also been studied by Moonens and Rosenblatt, see \cite{Monv.av.plane.MR}.\\

There appears to be a strong analogy between ergodic averages and the differentiation of integrals. Indeed, the limit (\ref{EQ - Lebesgue diff thm}) can be viewed as an ergodic analogue of the Lebesgue differentiation theorem, while the limit in (\ref{EQ - strong maximal theorem}) evokes the strong differentiation of integrals due to Jessen, Marcinkiewicz, and Zygmund (see \cite{Jessen1935}). In the theory of differentiation of integrals, the case of averaging over rectangles that make angles with the horizontal axis in the plane has already been explored. \\
\indent However, as far as we know, the case where the rectangle $\{R_n\}_{n \in \NN}$ make some angle with the horizontal axis has not been studied yet in the ergodic context. The aim of this paper is to study the influence of the directions in which the rectangles are oriented on convergence results in the ergodic setup.\\

Let us introduce some terminology. If $1 \leq p \leq \infty$, we say that an averaging process $\{R_n\}_{n \in \NN}$ is a \textit{$p$-good averaging process according to $S$ and $T$}, or simplier a \textit{$p$-good process}, if the averages $M_nf$ converge pointwise almost everywhere when $n \to \infty$, for every $f \in L^p (X)$. If there exists a function $f \in L^p(X)$ such that the averages $M_nf$ fail to converge almost everywhere, then we say that $\{R_n\}_{n \in \NN}$ is a \textit{$p$-bad averaging process according to $S$ and $T$}, or just a \textit{$p$-bad process}. In fact, as we shall see, being a $p$-good or $p$-bad process according to $S$ and $T$ doesn't depend on $S$ and $T$, hence we can just use the terminology $p$-good process and $p$-bad process.\\
\indent Recall that convergence almost everywhere of the averages $M_n$ in $L^p(X)$ is deeply related to weak boundedness properties of the corresponding maximal operator $M^*$, defined for a given function $f \in L^p(X), 1\leq p \leq \infty$, by
\[ M^*f := \sup_{n \in \NN} M_n|f|. \]
We say that $M^*$ is of weak-type $(p,p)$, $1 \leq p < \infty$, if there exists a constant $C > 0$ such that for any $f \in L^p(X)$ and any $\lambda > 0$ one has
\begin{align}
	\mu  (\{ x \in X : M^*f(x) > \lambda \}) \leq C  \left ( \frac{\|f\|_{L^p(X)}}{\lambda} \right )^p.
\end{align}
Furthermore, the operator $M^*$ is said to satisfy a \textit{Tauberian condition} if for any $0 < \lambda < 1$, there exists a constant $C_\lambda > 0$ such that for any measurable set $A \subset \AAA$ satisfying $0 < \mu(A) < \infty$, one has
\begin{align}\label{EQ - weak-type (infty,infty)}
	\mu (\{ x \in X : M^*\chi_A(x) > \lambda \}) \leq C_\lambda \mu(A).
\end{align}
Note that satisfying a Tauberian condition is implied by being of weak-type $(p,p)$, for some $1 \leq p < \infty$. Thus, in this context, satisfying a Tauberian condition can be seen as the ``weakest'' weak boundedness property for a maximal operator.
The link between the convergence almost everywhere of $M_nf$ and the boundedness properties of $M^*$ is known as the \textit{Sawyer-Stein principle}. Let us cite a version that will be useful to us. Its proof is postponed to Appendix \ref{APPENDIX - Sawyer Stein}.

\begin{thm}\label{THM - Sawyer Stein principle}
Let $1 \leq p < \infty$. The following are equivalent.
\begin{enumerate}[label=(\roman*)]
	\item Given any $f \in L^p(X)$, the limit
	\[ \lim_{n \to \infty} M_n f(x) = \lim_{n \to \infty} \frac{1}{\# (R_n \cap \ZZ^2)} \sum_{(i,j) \in R_n} f(S^iT^jx) \]
	exists for almost every $x \in X$.
	\item The maximal operator $M^*$ is of weak-type $(p,p)$.
\end{enumerate}
\end{thm}

This theorem can be reformulated as follows: a process $\{R_n\}_{n \in \NN}$ is a $p$-good process if and only if the associated maximal operator $M^*$ is of weak-type $(p,p)$. The same principle holds in various contexts of operator convergence. For instance, a Sawyer-Stein principle is also usable in the theory of differentiation of integrals. We will establish several ``transfer'' principles, allowing us to deduce ergodic results from results in the differentiation of integrals (see Appendix \ref{APPENDIX - transfer lemma}).\\

In \cite{MR4603294}, D'Aniello, Gauvan, Moonens and Rosenblatt constructed a differentiation processes along certain directional rectangles that do not converge almost everywhere, even for functions in $L^\infty(\RR^2)$. Let us briefly recall the context. If $\{Q_n\}_{n \in \NN}$ is a sequence of tilted rectangles in the plane centered at the origin such that $\mathrm{diam}(Q_n) \to 0$ as $n \to \infty$, we denote for $x \in \RR^2$
\[ T_nf(x) := \frac{1}{|Q_n|} \int_{Q_n+x} f, \] 
the average of a function $f \in L^1_{\mathrm{loc}}(\RR^2)$ over the rectangle $Q_n+x$. As mentioned above, convergence almost everywhere of such operators is linked to the boundedness property of the maximal operator $T^*$, defined as
\[ T^*f := \sup_{n \in \NN} T_n|f|, \]
for $f \in L^1_\mathrm{loc}(\RR^2)$.  The maximal operator $T^*$ is said to satisfy a Tauberian condition if for any $0 < \lambda < 1$ and any Borel set $B \subset \RR^2$ with finite Lebesgue measure, there exists a constant $C_\lambda > 0$ such that:
\[ | \{ x \in \RR^2 : T^* \chi_B (x) > \lambda \} | \leq C_\lambda |B|. \]
In their paper \cite{MR4603294}, the authors constructed a sequence $\{Q_n\}_{n \in \NN}$ of directional rectangles such that the maximal operator $T^*$ does not satisfy a Tauberian condition. One of the aims of this paper is to use our ``transfer'' results to obtain a process $\{R_n\}_{n \in \NN}$ such that the associated maximal ergodic operator $M^*$ does not satisfy a Tauberian condition either. Consequently, this process provides example of $p$-bad process, for any $1 \leq p < \infty$, in the ergodic context.\\
\indent On the other hand, we will also present some $p$-good processes $\{R_n\}_{n \in \NN}$ composed of \textit{directional} rectangles in the plane. \\
\indent Let us now formalize the definitions given above and state precisely our results.

\section{Results}

We consider a sequence $\{R_n\}_{n \in \NN} \subset \RR^2$ of rectangles containing the point $(0,0)$ and such that the sequence $\{l(R_n)\}_{n \in \NN}$ of their shortest sides tends to infinity as $n \to \infty$. Such a sequence is called a \textit{process}. Fix, as before, $S,T : X \to X$ a pair of ergodic, commuting, invertible, measure preserving, transformations on a probability space $ (X,\AAA,\mu)$, such that
\[\mu (\{x \in X : S ^iT^jx=x\})=0\]
for all $(i,j) \in \ZZ^2$, $(i,j) \neq (0,0)$. We wonder if the averages\label{PAGE - definitions process}
\[ M_nf = \frac{1}{\# (R_n \cap \ZZ^2)} \sum_{(i,j) \in R_n} f(S^iT^j) \]
converge almost everywhere as $n \to \infty$, for all $f : X \to \RR$ belonging to a given function space. We say that a process $\{R_n\}_{n \in \NN}$ is:

\begin{itemize}\label{PAGE - definitions of good and bad process}
	\item a \textit{$p$-good process} in case of the averages $M_nf$ converge almost everywhere for all $f \in L^p(X)$;
	\item a \textit{$p$-bad process} otherwise.
\end{itemize}

We define the \textit{slope} of a rectangle as the tangent of the angle formed between its longest side and the horizontal axis. To avoid ambiguity when the rectangle is a square, in this case we define its slope as the tangent of the smallest angle formed between any of its sides and the horizontal axis. Recall that a sequence of nonnegative real numbers $\{u_k\}_{k \in \NN}$ is \textit{lacunary} if there exists $\lambda \in (0,1)$ such that
\[ \fa k \in \NN, \quad  u_{k+1} \leq \lambda u_k.\]
As in the introduction, we consider the maximal operator $M^*$ associated to the averages $M_n$, given by:
\[ M^*f = \sup_{n \in \NN} M_n|f| \]
for $f \in L^p(X)$, $1 \leq p \leq \infty$. We recall what it means for $M^*$ to be of weak-type and to satisfy a Tauberian condition.\\
\indent For $1 \leq p < \infty$, we say that the maximal operator $M^*$ is of \textit{weak-type} $(p,p)$ if there exists a constant $C > 0$ such that for any $f \in L^p(X)$ and any $\lambda > 0$ one has
\begin{align*}
	\mu  (\{ x \in X : M^*f(x) > \lambda \}) \leq C  \left ( \frac{\|f\|_{L^p(X)}}{\lambda} \right )^p.
\end{align*}
We say that the operator $M^*$ satisfies a Tauberian condition if for any $0 < \lambda < 1$, there exists a constant $C_\lambda > 0$ such that for any measurable set $A \subset \AAA$ satisfying $0 < \mu(A) < \infty$, one has
\begin{align*}
	\mu (\{ x \in X : M^*\chi_A(x) > \lambda \}) \leq C_\lambda \mu(A).
\end{align*}
Our first result is the following: lacunary sets of slopes define good processes.

\begin{thm}\label{THM - good set}
	Let $\Omega \subset (0,1)$ be a lacunary sequence converging to $0$ and let $\{R_n\}_{n \in \NN}$ be a process such that for any $n \in \NN$, the slope of $R_n$ is in $\Omega$. Then $M^*$ is of weak-type $(p,p)$ for any $ 1 < p < \infty$.
\end{thm}

\begin{rmq}
	As mentioned before, this also implies that $M^*$ satisfies a Tauberian condition.
\end{rmq}

By the Sawyer-Stein principle, we have the following immediate corollary.

\begin{cor}\label{COR - good set}
		Let $\Omega \subset (0,1)$ be a lacunary sequence converging to $0$ and let $\{R_n\}_{n \in \NN}$ be a process such that for any $n \in \NN$, the slope of $R_n$ is in $\Omega$. Then for any $1 < p < \infty$, $\{R_n\}_{n \in \NN}$ is a \textit{$p$-good process}.
\end{cor}

\begin{rmq}
	In fact, the result holds if $\Omega$ is finitely lacunary, see \cite{SjogrenSjolin1981} for definitions.
\end{rmq}

\begin{rmq}
	Since one has $L^\infty(X) \subset L^p(X)$ for every $1 < p < \infty$, it follows that, under the same assumptions as in Corollary \ref{COR - good set}, the process $\{R_n\}_{n \in \NN}$ is also $\infty$-good.
\end{rmq}

On the other hand, processes along directional rectangles can also be bad. As announced, we will adapt a construction presented in \cite{MR4603294} in a differentiation setting into an ergodic context. More precisely, we have the following theorem.

\begin{thm}\label{THM - bad set}
	Let $\Omega = \{ u_k^{-1}\}_{k \in \NN^*}$, where $\{u_k\}_{k \in \NN^*}$ is an increasing sequence of nonnegative real numbers that satisfies the following two conditions (we let $u_0:=0$):
	\begin{itemize}
		\item[$(i)$] there exists a constant $c > 0$ such that for all $k \in \NN^*$, one has
		\begin{align}\label{EQ - hypothèse 1+b2 >= c(bk-bk-1)}
			1+u_{k-1}^2 \geq c(u_k-u_{k-1})^2 ;
		\end{align}
		\item[$(ii)$] one has
		\begin{align}
			\sup_{n\in\NN}\sup_{1\leq l\leq n} \left(\frac{u_{n+2l}-u_{n+l}}{u_{n+l}-u_n}+\frac{u_{n+l}-u_n}{u_{n+2l}-u_{n+l}}\right)<\infty.
		\end{align}
	\end{itemize}
	Then, there exists a process $\{R_k\}_{k \in \NN}$ such that:
	\begin{itemize}
		\item[$(i)$] for each $k \in \NN^*$, the slope of the rectangle $R_k$ is $\frac{1}{u_k}$;
		\item[$(ii)$] $\{R_n\}_{n \in \NN}$ is a $p$-bad process for all $1 \leq p < \infty$.
	\end{itemize}
	Moreover, any other process $\{\widetilde{R}_k\}_{k \in \NN}$ defined by $\widetilde{R}_k := \delta_k R_k$ with $\delta_k > 1$, is a $p$-bad process for all $1 \leq p < \infty$.
\end{thm}

\begin{exmpl}
For $s >0$, the set of directions $\{ k^{-s} : k \in \NN^* \}$ satisfies the condition of Theorem \ref{THM - bad set}.
\end{exmpl}

\begin{rmq}
	If $\{ R_n\}_{n \in \NN}$ is a $p$-good process for $1 \leq p < \infty$, it can be shown using standard arguments that the limit of the averages $M_nf$ for $f \in L^p(X)$ is constant almost everywhere and
	\[ \lim_{n \to \infty} M_n f(x) = \int_X f \d \mu \]
	for almost every $x \in X$.
\end{rmq}

\section{Main tools}

We introduce two key tools to address our problems: a comparison between the volume of a set and the number of integer points it contains, and a ``transfer'' result, as outlined in the introduction.

\subsection{Comparison between volume and cardinality}\label{SUBSECTION - Comparison between volume and cardinality}

The following lemma, although quite simple, will be essential to transfer all the volume estimates of subsets of $\RR^2$ to cardinality estimates in $\ZZ^2$, especially estimates used in the construction of \cite{MR4603294} in $\RR^2$.

\begin{lem}\label{LEM - Estimation volume - cardinal}
	Let $ E \subset \RR^2$ be a compact convex subset with piecewise smooth boundary. Then, there exists a scale $ \delta_0 > 1 $ such that for any $ \delta > \delta_0 $ 
	\[ \frac{1}{2} \leq \frac{\# ( \delta E \cap \ZZ^2 )}{|\delta E|} \leq \frac{3}{2} . \]
\end{lem}

\begin{proof}
	Let $ Q(x,r) $ be the cube centered at $ x \in \RR^2 $ with side-length $ r \geq 0 $, and define the sets
	\[E^- :=  \{ x \in E : d(x,\partial E) > \sqrt{2} \}, \quad E^+ := E \cup \{x \in \RR^2 : d(x,\partial E) \leq \sqrt{2}\}. \]
	Then, for $ \delta > 1 $, one has
	\[ \delta E^- \subset \bigcup_{x \in \delta E \cap \ZZ^2} Q(x,1) \subset \delta E^+.\]
	This implies that
	\[ |\delta E^-| \leq \left | \bigcup_{x \in \delta E \cap \ZZ^2} Q(x,1) \right | = \# \left ( \delta E \cap \ZZ^2 \right ) \leq |\delta E^+|.\]
	By connectedness, if $ \delta E \cap \mathbb{Z}^2 $ is empty, then so is $ \delta E^-$.  
	Since $ |\delta E^\pm| = |\delta E| \pm O(\delta) $	and $ |\delta E| = \delta^2 |E|$, we finally obtain
	\[ \frac{\# \left ( \delta E \cap \ZZ^2 \right ) }{|\delta E|} = 1 + O \left ( \frac{1}{\delta} \right ),\]
	which concludes the proof.
\end{proof}

\begin{rmq}
	For $ d \geq 2 $, if $ E $ is a subset of $\RR^d $ with the same assumptions as before, then the same proof gives us $ |\# (\delta E \cap \ZZ^d) - \mathrm{Vol}(\delta E)| = O(\delta^{d-1}) $.
\end{rmq}

If the set $E$ is a rectangle, we have a more precise statement, since we can explicitly compute $|E^-|$ and $|E^+|$. The following refinement will be useful in the future.

\begin{lem}\label{LEM - Estimation volume - cardinal pour les process}
	Let $\{R_n\}_{n \in \NN}$ be a process in $\RR^2$. Then, there exists an integer $N \in \NN$ such that for any $n \geq N$, one has
	\[ \frac{1}{2} \leq \frac{\# ( R_n \cap \ZZ^2 )}{|R_n|} \leq \frac{3}{2} . \]
\end{lem}

\begin{proof}
	Let $n \in \NN$. We define $R_n^+$ as the rectangle with the same center and orientation as $R_n$, obtained by extending each side by $\sqrt{2}$ at both ends. Similarly, define $R_n^-$ by shortening each side by $\sqrt{2}$ at both ends. By taking $n$ large enough, we can assume that $l(R_n) > 2\sqrt{2}$. Elementary computations lead to
	\[ \begin{cases}
		|R_n^-| = (l(R_n)-2\sqrt{2})(L(R_n)-2\sqrt{2}) \\
		|R_n^+|= (l(R_n)+2\sqrt{2})(L(R_n)+2\sqrt{2})
	\end{cases} \]
	where $L(R_n)$ denotes the length of the largest side of the rectangle $R_n$. As before, one has
	\[ R_n^- \subset \bigcup_{x \in R_n \cap \ZZ^2} Q(x,1) \subset R_n^+ \]
	and
	\[ (l(R_n)-2\sqrt{2})(L(R_n)-2\sqrt{2}) \leq \# (R_n \cap \ZZ^2) \leq (l(R_n)+2\sqrt{2})(L(R_n)+2\sqrt{2}).\]
	This yields
	\[ \frac{(l(R_n)-2\sqrt{2})(L(R_n)-2\sqrt{2})}{l(R_n)L(R_n)} \leq \frac{\# (R_n \cap \ZZ^2)}{|R_n|} \leq \frac{(l(R_n)+2\sqrt{2})(L(R_n)+2\sqrt{2}))}{l(R_n)L(R_n)}. \]
	By taking the limit as $n \to \infty$, we get
	\[ \frac{\# (R_n \cap \ZZ^2)}{|R_n|} \underset{n \to \infty}{\lra}1, \]
	which proves the lemma.
\end{proof}

\subsection{Transfer lemmas}

By identifying the pair $(k,l) \in \ZZ^2$ and the transformation $S^kT^l : X \to X$, the study of ergodic averages can be reduced to the study of translation averages on $\ZZ^2$. As we shall see, studying averages in $\ZZ^2$ can be easier than studying ergodic averages for two reasons: we can rely on elementary geometric arguments, and we can use some well-known results from the theory of differentiation of integrals. For $ \varphi \in l^\infty(\ZZ^2) $, $ (k,l) \in \ZZ^2 $ and $ n \in \NN^* $ we define
\[ A_n\varphi(k,l) = \frac{1}{\# (R_n \cap \ZZ^2)} \sum_{(i,j) \in R_n} \varphi(k+i,l+j) \quad \text{and} \quad A^*\varphi(k,l) = \sup_{n \in \NN^*} A_n|\varphi|(k,l).\]
We are going to study the connection between this maximal function and the maximal function $M^*$, defined as
\[ M^*f = \sup_{n \in \NN} M_n|f| \]
for $f \in L^p(X)$, $1 \leq p \leq \infty$, and where as before
\[ M_nf = \frac{1}{\# (R_n \cap \ZZ^2)} \sum_{(i,j) \in R_n} f(S^iT^j). \]

\begin{lem}\label{LEM - Transfert lemma 1<p<infty}
	Let $1 \leq p < \infty$.	The following are equivalent.
	\begin{enumerate}[label=($\roman*$)]
		\item The maximal operator $M^*$ is of weak-type $(p,p)$.
		\item The maximal operator $A^*$ is of weak-type $(p,p)$ in $l^p(\ZZ^2)$, \textit{i.e.} there exists a constant $C > 0$ such that for all $\lambda > 0$ and for all $\varphi \in l^p(\ZZ^2)$
		\begin{align}\label{EQ - défintion du weak-type dans lp(Z2)}
			\# \{ (k,l) \in \ZZ^2 : A^*\varphi (k,l) \geq \lambda\} \leq C \left( \frac{\| \varphi \|_{l^p(\ZZ^2)}}{\lambda} \right )^p .
		\end{align}
	\end{enumerate}
\end{lem}

We also have a transfer result about operators satisfying a Tauberian condition.

\begin{lem}\label{LEM - Transfert lemma Tauberian condition}
	The following are equivalent.
	\begin{enumerate}[label=($\roman*$)]
		\item The maximal operator $M^*$ satisfies a Tauberian condition.
		\item The maximal operator $A^*$ satisfies a Tauberian condition, \textit{i.e.} for any $\lambda > 0$ there exists a constant $C_\lambda > 0$ such that for all sets $E \subset \ZZ^2$ satisfying $0 < \# E < \infty$, one has
		\begin{align}\label{EQ - défintion du weak-type dans l infty(Z2)}
			\# \{ (k,l) \in \ZZ^2 : A^*\chi_E (k,l) \geq \lambda\} \leq C_\lambda \# E.
		\end{align}
	\end{enumerate}
\end{lem}

\begin{rmq}
	As announced in the introduction, we observe that the weak-type inequalities for $M^*$, and thus all the convergence results via Sawyer-Stein principle, do not depend on $S$ and $T$.
\end{rmq}

The proofs of the two above lemmas are given in Appendix \ref{APPENDIX - transfer lemma}. We can now move to the proofs of Theorem \ref{THM - good set} and Theorem \ref{THM - bad set}.

\section{Proof of Theorem \ref{THM - good set}}

Given a set of direction $\Omega \subset \RR$, we denote $\BBB_\Omega$ the collection of all the rectangles in the plane with a slope in $\Omega$. The \textit{directional maximal operator} associated to $\Omega$ is denote $\MMM^*_\Omega$ and defined by
\begin{align}\label{EQ - Définition de maximal directional operator}
	\MMM^*_\Omega f(x) := \sup_{x \in R \in \BBB_\Omega} \frac{1}{|R|} \int_R |f|
\end{align}
for $f \in L^1_\mathrm{loc}(\RR^2)$. If $\Omega$ is a lacunary sequence converging to $0$, we know that $\MMM_\Omega$ is bounded on $L^p(\RR^2)$ for $1 < p \leq \infty$. This result is due to Strömberg and C\'ordoba and Fefferman (for $p > 2$), and Nagel, Stein, and Wainger extended this result one year later to any $1 < p \leq \infty$ (see \cite{Stromberg1976}, \cite{MR476977} and \cite{0b992c85-8c1b-372f-a8bb-0efd679eb1e5}). Then, it was shown by Sjögren and Sjölin in \cite{SjogrenSjolin1981} that the results holds if $\Omega$ is finitely lacunary. Our strategy here is to compare the directional maximal operator $\MMM_\Omega^*$ and its discrete version, defined by
	\[ A^*_\Omega \varphi (k,l) := \sup_{\substack{R \in \BBB_\Omega \\ \# (R \cap \ZZ^2) \geq 1}} \frac{1}{\# (R \cap \ZZ^2)} \sum_{(i,j) \in R} |\varphi (i+k,j+l)| \]
for $\varphi \in l^\infty(\ZZ^2)$ and $(k,l) \in \ZZ^2$. As noted in section \ref{SUBSECTION - Comparison between volume and cardinality}, such a comparison only makes sense when the rectangles are large enough. This is not a problem, since we can always focus on the tail of the process. It's the point of the next lemma. The proof is standard and is left to the reader.

	\begin{lem}\label{LEM - truncated max op bounded iff max op is bounded}
		Let $N \in \NN$. Define the $N$-truncated discrete maximal operator by
		\[ A^*_N\varphi := \sup_{n \geq N} A_n |\varphi| \]
		for $\varphi \in l^\infty(\ZZ^2)$. Then if $1 < p < \infty$, $A_N^*$ is of weak-type $(p,p)$ if and only if the maximal operator $A^*$ is of weak-type $(p,p)$.
	\end{lem}

We can now move to the proof of Theorem \ref{THM - good set} whose statement is recalled below. Definition of process is given on page \pageref{PAGE - definitions process}.

\begin{thm*}
	Let $\Omega \subset (0,1)$ be a lacunary sequence converging to $0$ and let $\{R_n\}_{n \in \NN}$ be a process such that for any $n \in \NN$, the slope of $R_n$ is in $\Omega$. Then, the maximal operator $M^*$ defined as:
	\[ M^*f = \sup_{n \in \NN} M_n |f| = \sup_{n \in \NN} \frac{1}{\# (\R_n \cap \ZZ^2)} \sum_{(i,j) \in R_n} |f(S^iT^j)|, \]
	is of weak-type $(p,p)$ for any $ 1 < p < \infty$.
\end{thm*}

\begin{proof}[Proof of Theorem \ref{THM - good set}.]
We fix a function $\varphi \in l^p(\ZZ^2)$, $1 < p < \infty$. We define a function $f_\varphi : \RR^2 \to \RR$ by
\[ \begin{tabular}{|cccl}
	$f_\varphi : $ & $\RR^2$ & $\lms$ & $\RR$ \\
	& $(x,y)$ & $\lms$ & $\varphi(\lfloor x \rfloor, \lfloor y \rfloor)$. 
\end{tabular} \]
The function $f_\varphi$ lies in $L^p(\RR^2)$ and $\| f_\varphi \|_{L^p(\RR^2)} = \| \varphi \|_{l^p(\ZZ^2)}$.
Our goal is to compare the two maximal functions $A^*_\Omega\varphi$ and $\MMM^*_\Omega f_\varphi$. Let $n \in \NN$ and $(k,l) \in \ZZ^2$. Then
\begin{align*}
\frac{1}{\# (R_n \cap \ZZ^2)} \sum_{(i,j) \in R_n +(k,l)} |\varphi (i,j)| &= \frac{1}{\# (R_n \cap \ZZ^2)} \sum_{(i,j) \in R_n +(k,l)} \int_{[k,k+1] \times [l,l+1]} |f_\varphi ( x  ,  y )| \d x \d y \\
&= \frac{1}{\# (R_n \cap \ZZ^2)} \int_{\widetilde{R}_n+(k,l)} |f_\varphi (x,y)| \d x \d y,
\end{align*}
where $\widetilde{R}_n := \{ (x,y) \in \RR^2 : (\lfloor x \rfloor , \lfloor y \rfloor) \in R_n \}$. Let $N \in \NN$ be such that for all $n \geq N$, one has
\[ l(R_n) \geq 1 \quad \text{and} \quad \# (R_n \cap \ZZ^2) \geq \frac{|R_n|}{2},\]
(such a $N$ exists by Lemma \ref{LEM - Estimation volume - cardinal pour les process}). We then assume that $n \geq N$. It's clear that there exists another rectangle $P_n$ oriented as $R_n$ such that $\widetilde{R}_n \subset P_n$ and such that $|P_n| \leq \alpha |R_n|$, with $\alpha > 1$ a constant independant on $n$, see Figure \ref{FIG - AOmega et MOmega}.\footnote{It suffices to take the rectangle with the same center and orientation, and enlarge it so that it extends $\sqrt{2}$ beyond its original size on all sides. Then, $\alpha=9+4 \sqrt{2}$ is suitable.}

\begin{figure}
	\centering
	\includegraphics[width=0.3\textwidth]{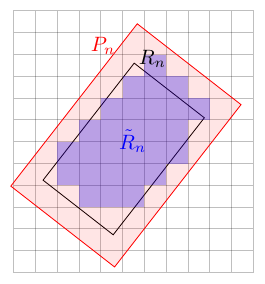}
	\caption{The set $\widetilde{R}_n$ and the rectangle $P_n$.}
	\label{FIG - AOmega et MOmega}
\end{figure}

Hence, we get:
\begin{align*}
	\frac{1}{\# (R_n \cap \ZZ^2)} \sum_{(i,j) \in R_n +(k,l)} |\varphi (i,j)| &\leq \frac{2}{|R_n|} \int_{P_n+(k,l)} |f_\varphi (x,y) \d x \d y \\
	&= \frac{2 \alpha}{|P_n|} \int_{P_n+(k,l)} |f_\varphi (x,y)| \d x \d y.
\end{align*}
By taking the supremum over $n \geq N$, we obtain
\[ A^*_N \varphi(k,l) \lesssim \MMM_\Omega^*f_\varphi (k,l), \]
where $A^*_N$ is the truncated maximal operator introduced in Lemma \ref{LEM - truncated max op bounded iff max op is bounded}. Then we have
\[ \| A_N^* \varphi \|_{l^p(\ZZ^2)}^p \lesssim \sum_{(k,l) \in \ZZ^2} |\MMM^*_\Omega f_\varphi(k,l)|^p = \int_{\RR^2} |\MMM^*_\Omega f_\varphi (\lfloor x \rfloor, \lfloor y \rfloor )|^p \d x \d y. \]
Let $(x,y) \in \RR^2$, and let $R \in \BBB_\Omega$ be a rectangle containing the point $(\lfloor x \rfloor , \lfloor y \rfloor)$. It's easy to see that there exists another rectangle $R' \in \BBB_\Omega$ containing $x$, call it $R$, such that $|\widetilde{R}| \lesssim |R|$. It comes that
\[ \frac{1}{|R|} \int_R |f_\varphi| \lesssim \frac{1}{|\widetilde{R}|} \int_{\widetilde{R}} |f_\varphi| \leq \MMM^*_\Omega f_\varphi(x,y) \]
and so $\MMM^*_\Omega f_\varphi (\lfloor x \rfloor , \lfloor y \rfloor) \lesssim \MMM^*_\Omega f_\varphi (x,y)$. Since $\MMM^*_\Omega$ is bounded on $L^p(\RR^2)$, we get
\[ \| A_N^* \varphi \|_{l^p(\ZZ^2)}^p \lesssim \int_{\RR^2} |\MMM_\Omega^*f_\varphi (x,y)|^p \d x \d y \lesssim \| f_\varphi \|_{L^p(\RR^2)}^p = \| \varphi \|_{l^p(\ZZ^2)}^p.\]
The operator $A_N^*$ is then bounded in $l^p(\ZZ^2)$, and so is $A^*$ by Lemma \ref{LEM - truncated max op bounded iff max op is bounded}. Using the transfer Lemma \ref{LEM - Transfert lemma 1<p<infty}, the proof is complete.
%
%
%
%
\end{proof}

\section{Proof of Theorem \ref{THM - bad set}}

As a first step, we exposed two geometric constructions presented in \cite{MR4603294} and we adapt it to $\ZZ^2$.

\subsection{Geometric construction 1}\label{SUBSECTION : Geometric construction 1}

In this section, we use Lemma \ref{LEM - Estimation volume - cardinal} to adapt the geometric construction presented in section $6$ of \cite{MR4603294}. We present here a discrete version of \cite[Lemma 6.1]{MR4603294}. For the sake of completeness, let us recall the geometric setup. We define the points $A=(1,0)$, $B=(0,b)$ and $C=(0,c) $ where $0 \leq c < b$, we denote $\Delta$ the full triangle $ABC$. Let $B',C'$ be the points defined by
\[ \overrightarrow{AB'} = \frac{3}{2} \overrightarrow{AB}, \quad \overrightarrow{AC'} = \frac{3}{2} \overrightarrow{AC}.\]
Let $\widetilde{P}$ be the smallest rectangle having $(AC')$ as one of its sides and containing the triangle $AB'C'$, and let $P$ be the rectangle obtained by translating $\widetilde{P}$ such that its center is the origin. Finally, we denote $V$ the full trapezium $BCC'B'$. See Figure \ref{FIGURE - The triangle Delta and the rectangle P}.

\begin{figure}
	\centering
	\includegraphics[width=0.8\textwidth]{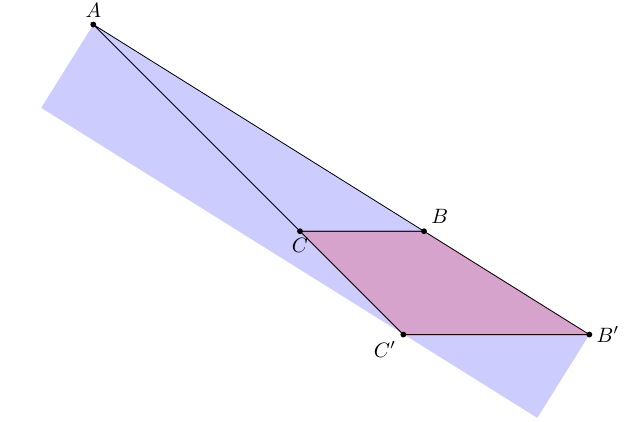}
	\caption{The triangle $\Delta$ and the rectangle $\widetilde{P}$.}
	\label{FIGURE - The triangle Delta and the rectangle P}
\end{figure}

In \cite[Lemma 6.1]{MR4603294}, the authors used elementary geometric arguments to prove the following.

\begin{lem}
	If $\alpha > 0$ is such that $AB \geq \alpha BC$, then for any $x \in V$ one has
	\[ |(x+P) \cap \Delta| \geq \frac{\min (\alpha,1)}{72}|P|. \]
\end{lem}

By Lemma \ref{LEM - Estimation volume - cardinal}, we directly have a discrete version of this estimate.

\begin{lem}\label{LEM - translation trapèze dans Z2}
There exists a scale $\delta = \delta_\Delta > 0$ depending on the triangle $\Delta$ such that for any $x \in V$ one has
	\[ \# \left ( \left ( (x+P) \cap \Delta \right ) \cap \ZZ^2\right ) \gtrsim \# \left ( \delta P \cap \ZZ^2 \right ). \]
\end{lem}

\subsection{Geometric construction 2}

In \cite{MR1650945}, Hare and Rönning introduced some generalized Perron trees constructions (we refer to \cite{Guzman-Differentiation-of-Integrals} for the basic construction). Given a set of directions $\Omega$, they gave necessary conditions on $\Omega$ ensuring that the maximal directional operator associated to $\Omega$ (see (\ref{EQ - Définition de maximal directional operator}) for the definition) is bounded on $L^p(\RR^2)$, for $1 < p \leq \infty$. The two authors introduced a quantity associated to $\Omega$, its \textit{Perron factor} $PF(\Omega)$. If $\Omega = \{u_n^{-1}\}_{n \in \NN}$ for an increasing sequence of slopes $\{u_n\}_{n \in \NN}$, then the \textit{Perron factor} of $\Omega$ is defined as 
\[ PF(\Omega) := \sup_{n\in\NN}\sup_{1\leq l\leq n} \left(\frac{u_{n+2l}-u_{n+l}}{u_{n+l}-u_n}+\frac{u_{n+l}-u_n}{u_{n+2l}-u_{n+l}}\right).\]
Hare and Rönning proved that if $\Omega = \{u_k^{-1}\}_{k \in \NN}$ for an increasing sequence of slopes $\{u_k\}_{k \in \NN}$ and if $PF(\Omega) < \infty$, then the operator $\MMM^*_\Omega$ is unbounded on $L^p(\RR^2)$, for all $1 < p \leq \infty$. Recently, D’Aniello, Gauvan, Moonens, and Rosenblatt adapted this construction and proved the following (\cite[Lemma 5.1]{MR4603294}).

\begin{lem}
	Assume that $\Omega:=\{u_k^{-1}\}_{k \in \NN^*}$, where $\{u_k\}_{k \in \NN^*}$ is an increasing sequence of nonnegative real numbers satisfying:
	\begin{equation}
		PF(\Omega) < \infty.
	\end{equation}
	Let $u_0=0$. Denote, for $k \in \NN^*$, by $\Delta_k$ the triangle in the plane having vertices $A:=(0,1)$, $B_{k-1}:=(u_{k-1}, 0)$ and $B_k:=(u_k,0)$. Under those assumptions, there exists a sequence $\{\varepsilon_n\}_{n \in \NN}$ of positive real numbers tending to $0$ and horizontal translations $\tau_k$, $k \in \NN^*$, in such a way that, for any $n \in \NN$, one has:
   \[ \left| \bigcup_{k=2^n}^{2^{n+1}-1} \tau_k \Delta_k\right|\leq \varepsilon_n \left| \bigcup_{k=2^n}^{2^{n+1}-1} \Delta_k\right|.\]
\end{lem}

Once again, we just have to use Lemma \ref{LEM - Estimation volume - cardinal} to obtain a discrete version of this result in $\ZZ^2$.

\begin{lem}[Dyadic blocks in discrete generalized Perron trees]\label{LEM - Perron tree dans Z2}
	Assume that $\Omega:=\{u_k^{-1}\}_{k \in \NN^*}$, where $\{u_k\}_{k \in \NN^*}$ is an increasing sequence of nonnegative real numbers satisfying:
	\begin{equation}
		PF(\Omega) < \infty.
	\end{equation}
	Let $u_0=0$. Denote, for $k \in \NN^*$, by $\Delta_k$ the triangle in the plane having vertices $A:=(0,1)$, $B_{k-1}:=(u_{k-1}, 0)$ and $B_k:=(u_k,0)$. Under those assumptions, there exists a sequence $\{\varepsilon_n\}_{n \in \NN}$ of positive real numbers tending to $0$, horizontal translations $\tau_k$, $k \in \NN^*$, and a sequence of scales $\{\delta_n\}_{n \in \NN} \subset (0,\infty)^\NN$ in such a way that, for any $n \in \NN$, one has:
	\[ \# \left( \left ( \delta_n \bigcup_{k=2^n}^{2^{n+1}-1} \tau_k \Delta_k \right ) \cap \ZZ^2\right) \lesssim \varepsilon_n \# \left( \left ( \delta_n \bigcup_{k=2^n}^{2^{n+1}-1} \Delta_k \right ) \cap \ZZ^2 \right).\]
\end{lem}

\begin{proof}
	We would like to apply Lemma \ref{LEM - Estimation volume - cardinal} directly, but the sets we consider are not all convex. While the set $\bigcup_{k=2^n}^{2^{n+1}-1} \Delta_k$ is convex, this is not the case for $\bigcup_{k=2^n}^{2^{n+1}-1} \tau_k \Delta_k$. Nevertheless, if $E = \bigcup_{k=2^n}^{2^{n+1}-1} \tau_k \Delta_k$, we can write
	\[ E = \bigcup_{j \in J} E_j \]
	where $J \subset \NN$ is finite, and the sets $E_j$ are disjoint sets satisfying the assumptions of Lemma \ref{LEM - Estimation volume - cardinal}. It then suffices to apply Lemma \ref{LEM - Estimation volume - cardinal} to the sets $E_j$, and to take the largest of the obtained scales.
\end{proof}

\subsection{Proof of Theorem \ref{THM - bad set}}

Let us recall Theorem \ref{THM - bad set}. The definition of a ($p$-bad) process is given on page \pageref{PAGE - definitions of good and bad process}.

\begin{thm*}
	Let $\Omega = \{ u_k^{-1}\}_{k \in \NN^*}$, where $\{u_k\}_{k \in \NN^*}$ is an increasing sequence of nonnegative real numbers that satisfies the following two conditions (we let $u_0:=0$):
\begin{itemize}
	\item[$(i)$] there exists a constant $c > 0$ such that for all $k \in \NN^*$, one has
	\begin{align*}
		1+u_{k-1}^2 \geq c(u_k-u_{k-1})^2 ;
	\end{align*}
	\item[$(ii)$] one has
	\begin{align*}
		PF(\Omega) < \infty.
	\end{align*}
\end{itemize}
Then, there exists a process $\{R_k\}_{k \in \NN}$ such that:
\begin{itemize}
	\item[$(i)$] for each $k \in \NN^*$, the slope of the rectangle $R_k$ is $\frac{1}{u_k}$;
	\item[$(ii)$] $\{R_n\}_{n \in \NN}$ is a $p$-bad process for all $1 \leq p < \infty$.
\end{itemize}
Moreover, any other process $\{\widetilde{R}_k\}_{k \in \NN}$ defined by $\widetilde{R}_k := \delta_k R_k$ with $\delta_k > 1$, is a $p$-bad process for all $1 \leq p < \infty$.
\end{thm*}

We start by proving the following proposition. Recall that we define the discrete maximal operator $A^*$ as:
\[ A^*\varphi = \sup_{n \in \NN^*} A_n|\varphi| \]
for $\varphi \in l^\infty(\ZZ^2)$ where
\[ A_n \varphi = \frac{1}{\# (R_n \cap \ZZ^2)} \sum_{(i,j) \in R_n} \varphi(\cdot +i,\cdot +j). \]

\begin{prop}\label{PROP - not of weak-type in Z2}
	Let $\Omega = \{ u_k^{-1}\}_{k \in \NN^*}$, where $\{u_k\}_{k \in \NN^*}$ is an increasing sequence of nonnegative real numbers that satisfies the conditions $(i)$ and $(ii)$ as in the above theorem.
	Then, there exists a collection of rectangles $\{R_k\}_{k \in \NN}$ with $\lim_{k \to \infty}l(R_k) = \infty$ such that:
	\begin{itemize}
		\item[$(i)$] for each $k \in \NN^*$, the slope of the rectangle $R_k$ is $\frac{1}{u_k}$;
		\item[$(ii)$] the maximal operator $A^*$ associated to $\{R_k\}_{k \in \NN}$ doesn't satisfy a Tauberian condition.
	\end{itemize}
\end{prop}

For the proof we of course follow and adapt the one exposed in \cite{MR4603294}.

\begin{proof}
We keep the notations introduced in the previous section. Given a triangle $\Delta_k$, we also denote respectively $P_k$ and $V_k$ the rectangle and the trapezium associated to $\Delta_k$ as in section \ref{SUBSECTION : Geometric construction 1}. We fix $n \in \NN^*$. Define the sets
\[ K^n := \bigcup_{k=2^n}^{2^{n+1}-1} \tau_k \Delta_k \quad \text{and} \quad V^n := \bigcup_{k=2^n}^{2^{n+1}-1} \tau_k V_k \]
Since for any $n \in \NN$, $V^n$ contains a similar copy of the triangle $\cup_{k=2^n}^{2^{n+1}-1}\Delta_k$ with ratio $\frac{1}{3}$, one has
\[ |V^n| \geq \frac{1}{9} \left | \bigcup_{k=2^n}^{2^{n+1}-1} \Delta_k \right | \]
so by Lemma \ref{LEM - Estimation volume - cardinal} and by Lemma \ref{LEM - Perron tree dans Z2} there exists a scale $ \delta^{(1)}_n $ such that
\[ \# \left ( \delta^{(1)}_n  V^n \cap \ZZ^2 \right )  \gtrsim \# \left ( \delta^{(1)}_n \bigcup_{k=2^n}^{2^{n+1}-1} \Delta_k \cap \ZZ^2 \right ) \]
and
\[ \frac{\# \left ( \delta_n^{(1)} \bigcup_{k=2^n}^{2^{n+1}-1} \Delta_k \cap \ZZ^2 \right )}{\# \left ( \delta_n^{(1)} \bigcup_{k=2^n}^{2^{n+1}-1} \tau_k \Delta_k \cap \ZZ^2 \right )} \gtrsim \frac{1}{\varepsilon_n}. \]
Let $ x \in V^n $. There exists $2^n \leq k \leq 2^{n+1}-1$ such that $ x \in \tau_k V_k $. For $k \geq 1$, note that by (\ref{EQ - hypothèse 1+b2 >= c(bk-bk-1)}) one has
\[ AB_{k-1} = \sqrt{1+u_{k-1}^2} \geq \sqrt{c}(u_k-u_{k-1}) = \sqrt{c}B_{k-1}B_k. \]
Hence, Lemma \ref{LEM - translation trapèze dans Z2} can be applied with triangles $\Delta_k$. Thus, there exists another scale $ \delta^{(2)}_n > 1$  such that
\[ \# \left( \left( \delta^{(2)}_n ((x+P_k) \cap \tau_k \Delta_k ) \right) \cap \ZZ^2 \right) \gtrsim \# (\delta^{(2)}_n P_k \cap \ZZ^2) \]
(taking the largest scale $\delta_{\Delta_k}$ associated to $\Delta_k$ as in Lemma \ref{LEM - translation trapèze dans Z2}, we can assume that $\delta_n^{(2)}$ depends only on $n$). Let $ t_0 > 0 $ be the implicit constant in the previous inequality, then
\begin{align*}
	\frac{\# \left( \left( \delta^{(2)}_n ((x+P_k) \cap K^n ) \right) \cap \ZZ^2 \right)}{\# (\delta^{(2)}_n P_k \cap \ZZ^2)}  \geq \frac{\# \left( \left( \delta^{(2)}_n ((x+P_k) \cap \tau_k \Delta_k ) \right) \cap \ZZ^2 \right)}{\# (\delta^{(2)}_n P_k \cap \ZZ^2)} \geq t_0.
\end{align*}
Let $\delta_n := \max \{ \delta_n^{(1)}, \delta_n^{(2)} \}$, and denote $\{A_k\}_{k \in \NN}$ the averaging operators associated to the sequence of rectangles $\{r_k P_k\}_{k \in \NN}$, where $r_k=\delta_n$ for $2^n \leq k \leq 2^{n+1}-1$. We can of course assume that $\delta_n$ tends to $\infty$ when $n \to \infty$. Then there exists another constant $t_0' > 0$ such that
\begin{align}\label{EQ - Proof of the theorem - 1}
	\frac{\# \left( \left( \delta_n ((x+P_k) \cap K^n ) \right) \cap \ZZ^2 \right)}{\# (\delta_n P_k \cap \ZZ^2)}  \geq \frac{\# \left( \left( \delta_n ((x+P_k) \cap \tau_k \Delta_k ) \right) \cap \ZZ^2 \right)}{\# (\delta_n P_k \cap \ZZ^2)} \geq t_0'.
\end{align}
The inequality (\ref{EQ - Proof of the theorem - 1}) rewrites
\[ A_k \chi_{\delta_n K^n}(\delta_n x) = \frac{1}{\# ( \delta_n P_k \cap \ZZ^2 )} \sum_{(i,j) \in \delta_n (x+P_k)} \chi_{\delta_n K^n}(i,j) \geq t_0.\]
This implies that $\delta_n V^n \cap \ZZ^2 \subset \{ y \in \ZZ^2 : A^* \chi_{\delta_n K^n}(y) \geq t_0 \}$. Using Lemma \ref{LEM - Perron tree dans Z2}, we finally get
\[ \frac{\# \{ x \in \ZZ^2 : A^* \chi_{\delta_n K^n}(x) \geq t_0 \}}{\# \left (\delta_n K^n \cap \ZZ^2 \right )} \geq \frac{\# \left ( \delta_n V^n \cap \ZZ^2 \right )}{\# \left (\delta_nK^n \cap \ZZ^2 \right )} \gtrsim \frac{\# \left ( \delta_n \bigcup_{k=2^n}^{2^{n+1}-1} \Delta_k \cap \ZZ^2 \right )}{\# \left ( \delta_n \bigcup_{k=2^n}^{2^{n+1}-1} \tau_k \Delta_k \cap \ZZ^2 \right )} \gtrsim \frac{1}{\varepsilon_n} \to \infty. \]
The last inequality shows that (\ref{EQ - défintion du weak-type dans lp(Z2)}) doesn't hold, so $A^*$ cannot satisfy a Tauberian condition.
\end{proof}

The proof of Theorem \ref{THM - bad set} is now trivial, it suffices to apply Proposition \ref{PROP - not of weak-type in Z2} and Lemma \ref{LEM - Transfert lemma Tauberian condition}. Sawyer-Stein principle finally gives us the result.

\appendix
\section{Proof of a Sawyer-Stein principle}\label{APPENDIX - Sawyer Stein}

We restate the Sawyer-Stein principle used previously (Theorem \ref{THM - Sawyer Stein principle}). Recall that the sequence $\{R_n\}_{n \in \NN}$ consists of rectangles in $\RR^2$ that contain the origin, and such that $l(R_n) \to \infty$ as $n \to \infty$. If $f :X \to \RR$ and $n \in \NN$, we still denote by
\[ M_nf = \frac{1}{\# (R_n \cap \ZZ^2)} \sum_{(i,j) \in R_n} f(S^iT^j). \]
the average of $f$ over $R_n$, and by $M^*f = \sup_{n \in \NN} M_n|f|$ the associated maximal operator.

\begin{thm*}
	Let $1 \leq p < \infty$. The following are equivalent.
	\begin{enumerate}[label=(\roman*)]
		\item Given any $f \in L^p(X)$, the limit
		\[ \lim_{n \to \infty} M_n f(x) = \lim_{n \to \infty} \frac{1}{\# (R_n \cap \ZZ^2)} \sum_{(i,j) \in R_n} f(S^iT^jx) \]
		exists for almost every $x \in X$.
		\item The maximal operator $M^*$ is of weak-type $(p,p)$.
	\end{enumerate}
\end{thm*}

In this section, we introduce a slightly different version of the maximal operator $M^*$. We define $M'^*$ as
\[ M'^*f := \sup_{n \in \NN} |M_nf| \]
for $f \in L^p(X)$, $1 \leq p \leq \infty$. If $1 \leq p < \infty$, we can show that $M'^*$ is of weak-type $(p,p)$ if and only if so is $M^*$. Therefore, it suffices to prove Sawyer-Stein principle for the operator $M'^*$. To simplify the notation, throughout this section, we denote $M^*$ the operator $M'^*$.

%
\begin{thm}\label{THM - Continuity principle for positive operators}
	Let $1 \leq p < \infty$. The following are equivalent.
	\begin{enumerate}[label=($\roman*$)]
		\item For every $f \in L^p(X)$, one has 
		\[ M^*f(x) < \infty \]
		for almost every $x \in X$.
		\item The maximal operator $M^*$ is of weak-type $(p,p)$.
	\end{enumerate}
\end{thm}

To prove Theorem \ref{THM - Continuity principle for positive operators}, we will need the following result, known as the \textit{Continuity Principle for Positive Operators}. We refer to Garsia \cite[p13]{Garsia1970TopicsIA}, for a proof. In the next statement, we use the notation $\AAA^* := \{ A \in \AAA : \mu(A) > 0\}$.

\begin{thm}\label{THM - Continuity principle for positive operators Garsia}
	Let $1 \leq p < \infty$, and let $\{T_n\}_{n \in \NN}$ be a sequence of positive linear operators, continuous in measure\footnote{Recall that a linear operator $T:X \to X$ is said to be \textit{continuous in measure} on $L^p(X)$ if for any sequence $\{f_n\}_{n \in \NN} \subset L^p(X)$ and for any $f \in L^p(X)$ such that $\|f_n-f\|_{L^p(X)} \to 0$ as $n \to \infty$, one has for any $\varepsilon > 0$:
		\[ \mu (\{ x \in X : |Tf_n(x)-Tf(x)|> \varepsilon \}) \underset{n \to \infty}{\lra} 0. \]} on $L^p(X)$. We define $T^*f:= \sup_{n \in \NN}| T_n f|$, for $f \in L^p(X)$. We assume that there exists $\EEE$, a family of measure-preserving transformations on $X$ satisfying the following:
	\begin{align}\label{EQ - mixing property}
		\fa A,B \in \AAA^*, \, \fa \alpha > 1, \, \ex E \in \EEE, \, \mu(A \cap E^{-1}B) \leq \alpha \mu(A) \mu(B);
	\end{align}
	\begin{align}\label{EQ - (weak) commute property}
		\fa E \in \EEE, \, \fa f \in L^p(X), \, T^*(f \circ E) \geq (T^*f) \circ E.
	\end{align}	
	Under these assumptions, the following are equivalent.
	\begin{enumerate}[label=($\roman*$)]
		\item For every $f \in L^p(X)$, one has
	\[  T^*f(x) < \infty \]
	for almost every $x \in X$.
		\item The maximal operator $ T^*$ is of weak-type $(p,p)$.
	\end{enumerate}
\end{thm}

\begin{rmq}
	A family of transformations $\EEE$ satisfying (\ref{EQ - mixing property}) is said to be \textit{mixing}.
\end{rmq}
We can now move to the proof of Theorem \ref{THM - Continuity principle for positive operators}
\begin{proof}[Proof of Theorem \ref{THM - Continuity principle for positive operators}]
	We would like to apply Theorem \ref{THM - Continuity principle for positive operators Garsia} to the sequence of operators $\{M_n\}_{n \in \NN}$. It is clear that these operators are positive, linear, and it is not hard to show that they are also continuous in measure (we can for instance apply Markov's inequality). Thus, it suffices to find a family $\EEE$ of measure preserving transformations satisfying (\ref{EQ - mixing property}) and (\ref{EQ - (weak) commute property}). We are going to show that the family $\EEE = \{T^n\}_{n \in \NN}$ is suitable. Property (\ref{EQ - (weak) commute property}) is easy to verify. Let us show that this family satisfies (\ref{EQ - mixing property}). Since $T$ is ergodic, we know that for any $A,B \in \AAA$, the following holds:
	\[ \frac{1}{n+1} \sum_{i=0}^n \mu(A \cap T^{-i}B) \underset{n \to \infty}{\lra} \mu(A)\mu(B) \]
	(see for example Petersen \cite[p56]{Petersen_1983}). Arguing by contradiction, if there exist $A,B \in \AAA^*$ and $\alpha > 1$ such that for any $n \in \NN$, one has $\mu(A \cap T^{-n}B) > \alpha \mu(A) \mu(B)$, we would have
	\[ \mu(A)\mu(B) = \lim_{n \to \infty}  \frac{1}{n+1} \sum_{i=0}^n \mu(A \cap T^{-i}B) \geq \lim_{n \to \infty} \frac{1}{n+1} \sum_{i=0}^n \alpha \mu(A) \mu(B) = \alpha \mu(A) \mu(B),\]
	which is a contradiction and concludes the proof. 
\end{proof}

To conclude the proof of the Sawyer-Stein principle, we have to show that for all $f \in L^p(X)$, $M^*f$ is finite almost everywhere if and only if the averages $M_nf$ converge almost everywhere. The ``if'' part is clear, we focus on the converse. The following result asserts that if $M^*f$ is finite almost everywhere for all $f \in L^p(X)$, then it suffices to show the convergence almost everywhere of the averages $M_nf$ only for $f$ in a dense class of functions in $L^p(X)$. See \cite[Theorem 1.1.1., p3]{Garsia1970TopicsIA} for the proof.

\begin{thm}\label{THM - weak-type implies that the set of convergence is closed}
	Let $1 \leq p < \infty$, and assume
		\begin{align}\label{EQ - M*f(x)<infty}
		M^*f(x) < \infty
	\end{align}
	for almost every $x \in X$, and all $f \in L^p(X)$. Denote
	\[ \DDD = \{ f \in L^p(X) : \{M_nf\}_{n \in \NN} \text{ converges almost everywhere} \}. \]
	Then $\DDD$ is closed in $L^p(X)$.
\end{thm}

We then have the immediate corollary. We use the notations of Theorem \ref{THM - weak-type implies that the set of convergence is closed}.

\begin{cor}\label{COR - Sawyer Stein principle proof}
If $\DDD$ contains a dense subspace of $L^p(X)$, then the averages $M_nf$ converge almost everywhere for all $f \in L^p(X)$.
\end{cor}
%
%

It is well known that the following subset
	\[ \CCC := \{ c + g-g \circ T : c \in \RR, g \in L^\infty(X)\} \]
is dense in $L^p(X)$ (see for instance \cite{CVmoving.av.BJR}). The last thing to verify is that if $f$ is in $\CCC_\infty$, then the averages $M_nf$ converge almost everywhere. Let $f=g-g \circ T$ with $g \in L^\infty(X)$ and let $x \in X$. One has
\begin{align*}
	M_nf(x) &= \frac{1}{\# (R_n \cap \ZZ^2)} \sum_{(i,j) \in R_n} g(S^iT^jx) - \frac{1}{\# (R_n \cap \ZZ^2)} \sum_{(i,j) \in R_n} g(S^iT^{j+1}x) \\
	&= \frac{1}{\# (R_n \cap \ZZ^2)} \sum_{(i,j) \in R_n} g(S^iT^jx) - \frac{1}{\# (R_n \cap \ZZ^2)} \sum_{(i,j) \in R_n+(0,1)} g(S^iT^jx) \\
	&= \frac{1}{\# (R_n \cap \ZZ^2)} \sum_{(i,j) \in \Delta_n} g(S^iT^jx),
\end{align*}
where $\Delta_n = R_n \Delta (R_n+(0,1))$ (here $\Delta$ denotes the symmetric difference). So, we finally have for almost every $x \in X$
\[ |M_nf(x)| \leq \frac{\# (\Delta_n \cap \ZZ^2)}{\# (R_n \cap \ZZ^2)} \| g \|_{L^\infty(X)}. \]

If $n$ is large enough, we saw that we can compare the cardinality of $\R_n \cap \ZZ^2$ (respectively $\Delta_n \cap \ZZ^2$) and the volume of $R_n$ (respectively $\Delta_n$). Furthermore, it is clear that
\[ | \Delta_n | \lesssim l(R_n)+L(R_n). \]
So if $n$ is large enough we get
	\[ \frac{\# (\Delta_n \cap \ZZ^2)}{\# (R_n \cap \ZZ^2)} \lesssim \frac{|\Delta_n|}{|R_n|} \lesssim \frac{l(R_n)+L(R_n)}{l(R_n)L(R_n)} \underset{n \to \infty}{\lra}0.\]
Then if $f \in \CCC$, we know that $M_nf$ converges almost everywhere. By Corollary \ref{COR - Sawyer Stein principle proof} and Theorem \ref{THM - weak-type implies that the set of convergence is closed}, the result is now proved for $1 \leq p < \infty$.

\section{Proof of the transfer lemmas}\label{APPENDIX - transfer lemma}

We start by proving Lemma \ref{LEM - Transfert lemma 1<p<infty}.

\begin{proof}[Proof of Lemma \ref{LEM - Transfert lemma 1<p<infty}]
	We first assume that $1 \leq p < \infty$. We firstly assume that $M^*$ is of weak-type $(p,p)$. Let $ N \in \NN^* $. By Rokhlin's lemma (see \cite{MR316680}), there exists a set $ \widetilde{\Omega}_N \in \AAA $ such that the sets $ \{ S^iT^j \widetilde{\Omega}_N\}_{0 \leq i,j \leq 2N-1} $ are disjoint and
	\[ \mu \left ( \bigcup_{i,j=0}^{2N-1} S^iT^j \widetilde{\Omega}_N \right ) \geq 1-\frac{1}{N^3}.\]
	Let $ \Omega_N :=  S^{N}T^{N} \widetilde{\Omega}_N$, and denote
	\[ X_N := \bigcup_{i,j=-N}^{N-1} S^iT^j \Omega_N = \bigcup_{i,j=-N}^{N-1} S^{i+N}T^{j+N} \widetilde{\Omega}_N = \bigcup_{i,j=0}^{2N-1} S^iT^j \widetilde{\Omega}_N . \]
	The sets $ \{ S^{i+N}S^{j+N}\widetilde{\Omega}_N\}_{-N \leq i,j \leq N-1} $ are pairwise disjoint and $ S,T $ are invertible measure preserving transformations, so we have $ \mu ( X_N ) = N^2 \mu (\Omega_N) \geq 1- \frac{1}{N^3}$. Hence we can write
	\[ X = X_N \cup X \backslash X_N, \quad \mu (X \backslash X_N) \leq \frac{1}{N^3}. \]
	
%
%
%
%
%
%

\begin{figure}
	\centering
	\includegraphics[width=0.5\textwidth]{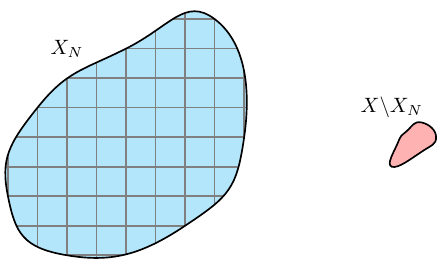}
	\caption{The space $ X=X_N \cup X \backslash X_N $.}
\end{figure}
	
	Let $ \lambda > 0 $ and $ \varphi \in l^p(\ZZ^2) $ be fixed. We are going to construct a function $ f \in L^p(X) $ linked to $ \varphi $. If $ x \in X_N $, there exists a unique couple $ (i,j) \in \{ -N, \ldots , N-1 \}^2$ such that $ x \in S^iT^j\Omega_N $. Hence, we can define an application $ \gamma_N : X \to \ZZ^2 $ as follow:
	\[\begin{tabular}{|crcl}
		$ \gamma_N : $ & $ X $ & $ \lra $ & $ \NN^2 $ \\
		& $ x $ & $ \lms $ & $ \begin{cases}
			(k,l) \quad \text{if $ x \in S^kT^l\Omega_N $ for some $ (k,l) \in \{ -N, \ldots , N-1 \}^2$,} \\
			(0,0) \quad \text{if $ x \in X \backslash X_N $.}
		\end{cases} $
	\end{tabular}\]
	We then define
	\[\begin{tabular}{|crcl}
		$ f_N : $ & $ X $ & $ \lra $ & $ \RR $ \\
		& $ x $ & $ \lms $ & $  f_N(x) := \varphi \circ \gamma_N(x). $
	\end{tabular}\]
	
	\begin{claim}\label{Claim 1}
		The function $ f_N $ defined above lies in $ L^p(X) $, and $ \| f_N \|_{L^p(X)} \leq  \left ( \frac{1}{N^3} + \mu (\Omega_N) \right )^{\frac{1}{p}} \| \varphi \|_{l^p(\ZZ^2)}$.
	\end{claim}
	
	\begin{proof}[Proof of the claim.]
		It's a direct calculation. One has
		\begin{align*}
			\| f_N \|_{L^p(X)}^p &= \int_{X \backslash X_N} |f_N|^p \d \mu + \int_{X_N} |f_N|^p \d \mu \\
			&= \int_{X \backslash X_N} |\varphi(0,0)|^p \d \mu + \sum_{-N \leq i,j \leq N-1} \int_{S^iT^j\Omega_N} |f_N|^p \d \mu \\
			&= |\varphi(0,0)|^p \mu (X \backslash X_N) + \mu(\Omega_N) \sum_{-N \leq i,j \leq N-1} |\varphi(i,j)|^p \\
			&\leq \frac{|\varphi(0,0)|^p}{N^3} + \mu(\Omega_N) \sum_{-N \leq i,j \leq N-1} |\varphi(i,j)|^p
		\end{align*}
		where the last line comes from the fact that $ \mu (S^iT^j\Omega_N)=\mu(\Omega_N) $ and $ \mu (X \backslash X_N) \leq \frac{1}{N^3} $. By writing $ |\varphi(0,0)|^p \leq \| \varphi \|^p_{l^p(\ZZ^2)} $ and $\sum_{-N \leq i,j \leq N-1} |\varphi (i,j)|^p \leq  \| \varphi \|^p_{l^p(\ZZ^2)}  $, we get the desired result.
	\end{proof}
	
	We denote
	\[ E_\lambda (\varphi) := \{(k,l) \in \ZZ^2 : A\varphi (k,l) \geq \lambda \} \quad \text{and} \quad F_\lambda (f_N) := \{ x \in X : Mf_N(x) \geq \lambda\}. \]
	By assumption we know that
	\begin{align}\label{Assumption}
		\mu (F_\lambda (f_N)) \leq C \left( \frac{\| f_N \|_{L^p(X)}}{\lambda} \right)^p 
	\end{align}
	holds for any $N \in \NN^*$. Since we are able to compare the norm of $ \varphi $ and the norm of $ f_N $, our goal is now to compare the quantities $ \# E_\lambda (\varphi) $ and $ \mu (F_\lambda (f_N)) $.
	We introduce some truncated versions of our maximal operators. For $m \in \NN^*$, we define $A^m := \sup_{n=1, \ldots , m} A_n | \cdot |$ and $M^m := \sup_{n=1, \ldots , m} M_n | \cdot |$. Let's also denote
	\[ E^m_\lambda(\varphi) := \{ (k,l) \in \ZZ^2 : A^m\varphi (k,l) \geq \lambda \} \quad \text{and} \quad F^m_\lambda(f_N):= \{ x \in X : M^mf_N(x) \geq \lambda \}. \]
	Note that since $M^m f \leq Mf$ for each $f \in L^1(X)$, $M^m$  is also of weak-type $(1,1)$ and the following holds for any $N \in \NN^*$:
	\begin{align}\label{Assumption truncated}
		\mu (F^m_\lambda (f_N)) \leq C \left( \frac{\| f_N \|_{L^p(X)}}{\lambda} \right)^p .
	\end{align}
	
	\begin{claim}\label{Claim 2}
		If $m \in \NN^*$, there exists an integer $N_m$ depending on $m$ such that
\[		\# \left (E^m_\lambda(\varphi) \cap \{-N_m, \ldots , N_m-1\} \right ) \leq \frac{\mu (F^m_\lambda (f_{N_m}))}{\mu(\Omega_{N_m})}.\]
	\end{claim}
	
	\begin{proof}[Proof of the claim]
		One has for any $N$
		\begin{align*}
			\mu (F^m_\lambda (f_N)) &= \mu (F^m_\lambda (f_N) \cap X \backslash X_N) + \sum_{-N \leq k,l \leq N-1} \mu \left (F^m_\lambda(f_N) \cap S^kT^l \Omega_N \right ) \\
			&\geq 0 + \sum_{(k,l) \in E_\lambda^m (\varphi) \cap \{-N, \ldots , N-1\} } \mu \left (F^m_\lambda(f_N) \cap S^kT^l \Omega_N \right ).
		\end{align*}
		Fix $ (k,l) \in E_\lambda^m (\varphi) \cap \{-N, \ldots , N-1\}$, and let $ x \in S^kT^l\Omega_N $. For $ 1 \leq n \leq m $ we have
		\begin{align*}
			A_n|\varphi|(k,l) &= \frac{1}{\# (R_n \cap \ZZ^2)} \sum_{(i,j) \in R_n} |\varphi(k+i,l+j)| \\
			&= \frac{1}{\# (R_n \cap \ZZ^2)} \sum_{\substack{(i,j) \in R_n \\  k+i,l+j \in \{ -N , \ldots , N-1 \} }} |\varphi(k+i,l+j)| \\
			&\hspace{5cm} + \frac{1}{\# (R_n \cap \ZZ^2)} \sum_{\substack{(i,j) \in R_n \\  k+i,l+j \notin \{ -N , \ldots , N-1 \} }} |\varphi(k+i,l+j)| \\
			&= \frac{1}{\# (R_n \cap \ZZ^2)} \sum_{\substack{(i,j) \in R_n \\   k+i,l+j \in \{ -N , \ldots , N-1 \} }} |f_N(S^iT^jx)| \\
			&\hspace{5cm} + \frac{1}{\# (R_n \cap \ZZ^2)} \sum_{\substack{(i,j) \in R_n \\  k+i,l+j \notin \{ -N , \ldots , N-1 \} }} |\varphi(k+i,l+j)|.
		\end{align*}
		Since the sets $\{R_n\}_{n=1, \ldots , m}$ are bounded, we can choose $N=N_m$ large enough such that
		\[ A_n | \varphi | (k,l) =  \frac{1}{\# (R_n \cap \ZZ^2)} \sum_{(i,j) \in R_n } |f_{N_m}(S^iT^jx)| = M_n |f_{N_m}|(x)\]
		for any $n \in \{1, \ldots , m\}$, so $M^m f_{N_m} (x) = A^m \varphi (k,l) \geq \lambda$ and $ S^kT^l \Omega_{N_m} \subset F^m_\lambda (f_{N_m}) $. The calculation is now easier:
		\begin{multline*}
			\mu (F^m_\lambda(f_{N_m}) ) \geq \sum_{(k,l) \in E_\lambda^m (\varphi) \cap \{-N_m, \ldots , N_m-1\}} \mu \left ( S^kT^l \Omega_{N_m} \right ) \\
			 = \mu ( \Omega_{N_m}) \# \left (E^m_\lambda (\varphi) \cap \{ -N_m , \ldots , N_m-1 \}\right ),
		\end{multline*}
		and Claim \ref{Claim 2} is proved.
	\end{proof}
	
	Let $m \in \NN$ and $N_m \in \NN^*$ be associated to $m$ as in the Claim \ref{Claim 2}. We can now bring back together the results of Claim \ref{Claim 1}, Claim \ref{Claim 2} and (\ref{Assumption truncated}), to obtain
	\begin{align*}
		\# \left (E_\lambda^m (\varphi) \cap \{-N_m, \ldots N_m-1 \} \right ) &\leq \frac{\mu (F^m_\lambda (f_{N_m}))}{\mu(\Omega_{N_m})} \\
		&\leq  \frac{C}{\mu(\Omega_{N_m})} \left( \frac{\| f_{N_m} \|_{L^p(X)}}{\lambda} \right)^p \\
		&\leq C \left ( 1 + \frac{1}{{N_m}^3 \mu(\Omega_{N_m})} \right ) \left( \frac{\| \varphi \|_{l^p(\ZZ^2)} }{\lambda}\right)^p  .
	\end{align*}
	The next step is to make $ m $ tend towards $ \infty $. We can assume that $N_m \to \infty$ when $m \to \infty$. On the one hand, we clearly have
	\[ \# E_\lambda (\varphi) = \lim_{m \to \infty} \# \left (E_\lambda^m (\varphi) \cap \{-N_m, \ldots N_m-1 \} \right ) . \]
	On the other hand $ N_m^2\mu(\Omega_{N_m}) \geq 1-\frac{1}{{N_m}^3}$, so that:
	\[\lim_{m \to \infty} \frac{1}{N_m^3\mu(\Omega_{N_m})} = 0.\]
	We finally have
	\[ \# E_\lambda (\varphi) \leq C  \left( \frac{\| \varphi \|_{l^p(\ZZ^2)} }{\lambda}\right)^p, \]
	which is what we wanted.\\
	
	Let's move on proving the converse. Assume that $A^*$ is of weak-type $(p,p)$. Let $K \in \NN$, and $f \in L^p(X)$. For $x \in X$, define
\begin{align}\label{EQ - transfer lemma - déf de phi x à partir de f}
		\varphi_x(k,l) := \begin{cases}
		f(S^kT^lx) \text{ if $|k|,|l| \leq K$,} \\
		0 \text{ otherwise.}
	\end{cases}
\end{align}
	For $n \in \NN$, define
	\[ i_\mathrm{max}(n) := \max_{(k,l) \in R_n} k \quad \text{and} \quad j_\mathrm{max}(n) := \max_{(k,l) \in R_n} l.\]
	For $n \in \NN$ and $(k,l) \in \ZZ^2$, one has
	\begin{align*}
		A_n |\varphi_x |(k,l) = \frac{1}{\# (R_n \cap \ZZ^2)} \sum_{(i,j) \in R_n}  |\varphi_x (k+i,l+j)| = \frac{1}{\# (R_n \cap \ZZ^2)} \sum_{\substack{(i,j) \in R_n \\   |k+i|,|l+j| \leq K }} |f(S^{k+i}T^{l+j}x)|.
	\end{align*}
	Given $m \in \NN$, for any $n \in \NN$ satisfying $|i_\mathrm{max}(n)|,  |j_\mathrm{max}(n)| \leq m$, and $k,l \in \ZZ^2$ satisfying $-K+m \leq k,l \leq K-m$, one has $|k+i|,|l+j| \leq K$ for all $(i,j) \in R_n$ and so
	\[ \sup_{n \, : \, |i_\mathrm{max}(n)|,  |j_\mathrm{max}(n)| \leq m} A_n|\varphi_x|(k,l) = \sup_{n \, : \, |i_\mathrm{max}(n)|,  |j_\mathrm{max}(n)| \leq m} M_n|f|(S^kT^lx). \]
	We note $\widetilde{A}^m := \sup_{n \, : \, |i_\mathrm{max}(n)|,  |j_\mathrm{max}(n)| \leq m} A_n$ and $\widetilde{M}^m := \sup_{n \, : \, |i_\mathrm{max}(n)|,  |j_\mathrm{max}(n)| \leq m} M_n$. Obverse that the truncated operator $\widetilde{A}^m$ is still of weak-type $(p,p)$. Let $\lambda > 0$. Define the set
	\[ S_\lambda := \{ (x,(k,l)) \in X \times \ZZ^2 : -K+m \leq k,l \leq K-m, \widetilde{M}^mf(S^kT^lx)> \lambda \}, \]
	and denote by $\mathrm{Count}$ the counting measure on $\ZZ^2$. By Fubini's theorem, we have:
	\[ \int_X \int_{\ZZ^2} \chi_{S_\lambda} \d \mathrm{Count} \d \mu = \int_{\ZZ^2} \int_X \chi_{S_\lambda} \d \mu \d \mathrm{Count}.\]
	The idea is then to estimate separately these two quantities.
	\begin{enumerate}
		\item On the one hand, \begin{align*}
			\int_X \int_{\ZZ^2} \chi_{S_\lambda} \d \mathrm{Count} \d \mu &= \int_X \# \{ (k,l) \in \ZZ^2 : -K+m \leq k,l \leq K-m, \widetilde{M}^mf(S^kT^lx) >  \lambda \} \d \mu(x) \\
			&= \int_X \# \{ (k,l) \in \ZZ^2 : -K+m \leq k,l \leq K-m, A^m\varphi(k,l) >  \lambda \} \d \mu(x). \\
			&\leq \int_X \# \{ (k,l) \in \ZZ^2 : A^m\varphi(k,l) >  \lambda \} \d \mu(x)
		\end{align*}
		We now use the fact that $\widetilde{A}^m$ is of weak-type $(p,p)$, this yields
		\begin{align*} 
			\int_X \int_{\ZZ^2} \chi_{S_\lambda} \d \mathrm{Count} \d \mu \leq \int_X  C \left (\frac{\| \varphi_x \|_{l^p(\ZZ^2)}}{\lambda}\right )^p \d \mu(x) = \frac{C}{\lambda^p} \int_X \sum_{-K \leq k,l \leq K} |f(S^kT^lx)|^p \d \mu(x).
		\end{align*}
		Since $S$ and $T$ are measure preserving, we finally have
		\begin{align}\label{EQ - transfert lemma Fubini 1}
			\int_X \int_{\ZZ^2} \chi_{S_\lambda} \d \mathrm{Count} \d \mu \leq \frac{C}{\lambda^p} (2K+1)^2 \| f \|^p_{L^p(X)}.
		\end{align}
		\item On the other hand, since $S$ and $T$ are invertible and measure preserving
		\begin{align*}
			\int_{\ZZ^2} \int_X \chi_{S_\lambda} \d \mu \d \mathrm{Count} &= \sum_{-K+m \leq k,l \leq K-m} \mu \left ( \left \{ x \in X : \widetilde{M}^mf(S^kT^lx) > \lambda \right \} \right ) \\
			&= \sum_{-K+m \leq k,l \leq K-m} \mu \left ( S^{-k}T^{-l}\left \{ y \in X : \widetilde{M}^mf(y) > \lambda \right \} \right ) \\
			&= \sum_{-K+m \leq k,l \leq K-m} \mu \left ( \left \{ y \in X : \widetilde{M}^mf(y) > \lambda \right \} \right ),
		\end{align*}
		and then
		\begin{align}\label{EQ - transfert lemma Fubini 2}
			\int_{\ZZ^2} \int_X \chi_{S_\lambda} \d \mu \d \mathrm{Count} = (2K-2M+1)^2 \mu \left ( \left \{ y \in X : \widetilde{M}^mf(y) > \lambda \right \} \right ).
		\end{align}
	\end{enumerate}
	Combining (\ref{EQ - transfert lemma Fubini 1}) and (\ref{EQ - transfert lemma Fubini 2}) we obtain
	\[ \mu \left ( \left \{ y \in X : \widetilde{M}^mf(y) > \lambda \right \} \right ) \leq \frac{C}{\lambda^p} \left (\frac{2K+1}{2K-2m+1} \right )^2 \| f \|^p_{L^p(X)}.\]
	Letting $K \to \infty$ in the previous inequality, we conclude that
	\[ \mu \left ( \left \{ y \in X : \widetilde{M}^mf(y) > \lambda \right \} \right ) \leq C \left( \frac{\| f \|_{L^p(X)}}{\lambda} \right)^p .\]
	Since the sets $\left ( \left \{ x \in X : \widetilde{M}^mf(x)>\lambda \right \} \right )_{m \in \NN}$ are increasing to $\{x \in X : M^*f(x) > \lambda\}$, we get that
	\[ \mu \left ( \left \{ y \in X : M^*f(y) > \lambda \right \} \right ) \leq C \left( \frac{\| f \|_{L^p(X)}}{\lambda} \right)^p, \]
	which prove the converse for $1 \leq p < \infty$.
\end{proof}

The proof of Lemma \ref{LEM - Transfert lemma Tauberian condition} is a direct consequence of the above.

\begin{proof}[Proof of Lemma \ref{LEM - Transfert lemma Tauberian condition}]
	We can follow exactly the same proof as above, we just need to verify that if $\varphi$ is an indicator function, then $f_N := \varphi \circ \gamma_N$ is also an indicator function. So let $E \subset \ZZ^2$ such that $\# E < \infty$. Then, for any $N \in \NN$ one has for all $x \in X$
	\[ f_N(x) = \chi_E(\gamma_N(x)) = \chi_{\widetilde{E}}(x),\]
	where $\widetilde{E} = \{ x \in S^kT^l \Omega_N : (k,l) \in E \cap \{-N, \ldots , N-1\}^2\}$. We do indeed obtain an indicator function. Let us focus on the converse. Let $A \in \AAA$ be a measurable set such that $0 < \mu (A) < \infty$. We fix $x \in X$, and we define $\varphi_x$ associated to $f = \chi_A$ as above (see (\ref{EQ - transfer lemma - déf de phi x à partir de f})). For $(k,l) \in \ZZ^2$, we have
	\[ \varphi_x(k,l) = \begin{cases}
		\chi_A (S^kT^lx) \text{ if $|k|,|l| \leq K$,} \\
		0 \text{ otherwise.}
	\end{cases}
	\]
	Then $\varphi$ is the indicator function of the set $\{ (k,l) \in \{-K, \ldots , K\}^2 : S^kT^lx \in A \}$.
\end{proof}

\section*{Acknowledgment}
We would like to thank the anonymous referee for their careful reading and helpful comments, which helped in improving the present paper. We also thank Laurent Moonens for his time, his valuable advice, and for introducing us to this subject.

\bibliographystyle{plain}
\bibliography{bibliographie}
%

\end{document}